\documentclass[11pt]{amsart}
\usepackage{graphicx,mathrsfs,setspace,fancyhdr,amsmath,fullpage,amssymb}
\usepackage[margin=2.5cm]{geometry}
\usepackage[alphabetic]{amsrefs}
\usepackage[colorlinks=true,linkcolor=blue,urlcolor=blue, citecolor=blue]{hyperref} 
\usepackage[all]{xy}
\usepackage{framed}
\usepackage{tikz}

\newtheorem{thm}{Theorem}[section]
\newtheorem*{thma}{Theorem A}
\newtheorem*{thm*}{Theorem}

\newtheorem*{thmb}{Theorem B}

\newtheorem*{corc}{Corollary C}
\newtheorem*{conj*}{Conjecture}
\newtheorem*{cor*}{Corollary}
\newtheorem{cor}[thm]{Corollary}
\newtheorem{conj}[thm]{Conjecture}
\newtheorem{lem}[thm]{Lemma}
\newtheorem{theorem}[thm]{Theorem}
\newtheorem{prop}[thm]{Proposition}

\theoremstyle{definition}
\newtheorem{defn}[thm]{Definition}

\theoremstyle{remark}
\newtheorem{rem}[thm]{Remark}
\newtheorem{rmk}[thm]{Remark}

\newtheorem{eg}[thm]{Example}
\numberwithin{equation}{section}

\makeatletter
\newcommand\xleftrightarrow[2][]{\ext@arrow 0099{\longleftrightarrowfill@}{#1}{#2}}
\def\longleftrightarrowfill@{\arrowfill@\leftarrow\relbar\rightarrow}
\makeatother

\def\mc{\mathscr}

\DeclareMathOperator{\Sym}{Sym}

\DeclareMathOperator{\Pic}{Pic}

\def\dim{\text{dim}}

\def\pr{\text{pr}}

\def\N{\mathbb{N}}
\def\P{\mathbb{P}}

\def\Z{\mathbb{Z}}

\def\cA{\mc{A}}

\def\cE{\mc{E}}
\def\cF{\mc{F}}
\def\cG{\mc{G}}

\def\cL{\mc{L}}
\def\cM{\mc{M}}

\def\cO{\mc{O}}

\def\cQ{\mc{Q}}

\def\cU{\mc{U}}

\newcommand{\PP}{\ensuremath{\mathbb{P}}}

\DeclareMathOperator{\Alb}{Alb}
\DeclareMathOperator{\alb}{alb}

\DeclareMathOperator{\ev}{ev}
\DeclareMathOperator{\ch}{ch}
\DeclareMathOperator{\td}{td}
\DeclareMathOperator{\NS}{NS}

 \newcommand{\equ}{\ensuremath{\,=\,}}
\newcommand{\dgeq}{\ensuremath{\,\geq\,}}
\newcommand{\dleq}{\ensuremath{\,\leq\,}}
\newcommand{\deq}{\ensuremath{\stackrel{\textrm{def}}{=}}}

\newcommand{\st}[1]{\ensuremath{\left\{ #1 \right\}   }}

\newcommand{\ssh}{simple semihomogeneous\ }

\newcommand{\scal}[1]{\ensuremath{\langle #1 \rangle}}

\renewcommand{\dim}{\ensuremath{\textrm{dim}\,}}
\newcommand{\rk}{\ensuremath{\textrm{rk}\,}}
\DeclareMathOperator{\Nef}{Nef}

\begin{document}

\title{Effective global generation on varieties with numerically trivial canonical class}
\author{Alex K\"{u}ronya}
\address{Alex K\"uronya, Institut f\"ur Mathematik, Goethe-Universit\"at Frankfurt, Robert-Mayer-Str. 6-10., D-60325 Frankfurt am Main, Germany}
\address{BME TTK Matematika Int\'ezet Algebra Tansz\'ek, Egry J\'ozsef u. 1., H-1111 Budapest, Hungary}
\email{\tt kuronya@math.uni-frankfurt.de}
\author{Yusuf Mustopa}
\address{Yusuf Mustopa, University of Massachusetts Boston, Department of Mathematics, Wheatley Hall, 100 William T Morrissey Blvd, Boston, MA 02125, USA}
\address{Max-Planck-Institut f\"{u}r Mathematik, Vivatsgasse 7, 53111, Bonn, Germany}
\email{Yusuf.Mustopa@umb.edu}

\begin{abstract}  
We prove a Fujita-type theorem for varieties with numerically trivial canonical bundle using properties of semihomogeneous bundles on abelian varieties.  We combine our results with work of Riess on compact hyperk\"{a}hler manifolds
and work of Mukai, Pareschi and Yoshioka to obtain effective global generation statements for certain moduli spaces of sheaves on abelian surfaces.  Among these is the statment that if $\cL$ is an ample line bundle on the Hilbert square $S^{[2]}$ of an abelian surface $S,$ then $\cL^{\otimes m}$ is globally generated for $m \geq 3.$
\end{abstract}


\maketitle

\section*{Introduction}

Fujita's conjectures on global generation and very ampleness have long been a strong driving force in birational  geometry. 
The conjectures say that for a polarized complex manifold $(X,\cL)$, the line bundle $\omega_X\otimes\cL^{\otimes m}$ should be basepoint-free if $m\geq \dim(X)+1$, and very ample whenever $m\geq \dim(X)+2$. The statement  is a classical consequence of Riemann--Roch if $\dim(X)=1$, and has been verified for $\dim(X)\leq 5$ in the case of global generation \cites{Reider,EL93,Kaw,Helmke,YZ}. There are further  partial results for the case of very ampleness when $X$ is a Calabi-Yau threefold \cite{GP} and more generally when $X$ has a nef canonical bundle \cite{MR}.  For arbitrary dimension, there exist strong global generation statements due to Angehrn and Siu \cite{AS} and Heier \cite{Heier}, whose bounds are nevertheless not linear in $\dim(X)$. 

While sharp for hyperplane bundles on projective spaces for instance, Fujita's conjecture is very far from the truth in general.  One notable class where global generation holds for much smaller powers is that of abelian varieties. Indeed, for a polarized abelian variety $(X,\cL)$ of any dimension,  the line bundle $\cL^{\otimes m}$ is globally generated for $m \geq 2$ and very ample for $m \geq 3$ by a theorem of Lefschetz. Similar statements include a result of Pareschi and Popa \cite{PP1}*{Theorem 5.1} to the effect  that if $(X,\cL)$ is a polarized smooth irregular variety whose Albanese morphism is finite onto its image, then $(\omega_X\otimes\cL)^{\otimes 2}$ is globally generated.  We mention the recent work \cite{Riess} on the base locus of $\cL^{\otimes 2}$ when $X$ is an irreducible holomorphic symplectic manifold.

Following this train of thought, we define for $0 \leq k \leq {\rm dim}(X)$ the $k-$th Fujita number $f_{\cE,k}$  of a coherent sheaf $\cE$ on a smooth projective variety as 
\[
f_{\cE,k} \deq \min \{m\geq 1\,|\, \omega_X\otimes \Sym^{m'}\cE \text{ is  globally generated in codimension $k$ for all $m'\geq m$} \}\ .
\]
Along the same lines we define the Fujita number  $f_{X,k}$ of a given smooth projective variety $X$  as 
\[
f_{X,k} \deq \max\{ f_{\cA,k}\,|\, \text{$\cA$ is an ample line bundle on $X$}\}\ ,
\]
while for a smooth fibration $\pi\colon X\to Y$ we set $f_{\pi,k} \deq \max \st{f_{X_y,k}\colon y\in Y}$.  Note that the sequences $\{f_{\cE,k}\}_{k}, \{f_{X,k}\}_{k}, \{f_{\pi,k}\}_{k}$ are all nondecreasing.  Going forward we write $f_{\cE,\dim{X}}$ as $f_{\cE}$ (similarly for $f_{X,\dim{X}}$ and $f_{\pi,\dim{X}}$).

Fujita's prediction for global generation can be phrased as $f_{X}\leq \dim(X)+1$.  The examples above suggest that $f_{X}$ may be rather strongly influenced by the Albanese dimension. Since there exist non-minimal $X$ with maximal Albanese dimension and $f_{X} \geq \dim(X) - 1$ (see Example \ref{ex:ab-var-blowup}), it is natural to restrict our attention to minimal varieties. 

The main goal  of our paper is an effective global generation result for minimal varieties of Kodaira dimension zero.  Recall that these are precisely the varieties whose canonical bundle is numerically trivial, and that their Albanese maps are \'{e}tale-trivial fibrations (Section 8, \cite{Ka}).

\begin{thma}[Theorem~\ref{thm:codim of base loci}]
	Let $X$ be a smooth projective variety with $K_{X} =_{\rm num} 0,$ and let $\cL$ be an ample line bundle on $X$.  Then for $1 \leq k \leq \dim X$ we have that
	\begin{enumerate}
		\item $\cL^{\otimes rs}$ is globally generated in codimension $k$ for 
		\[
		r \geq f_{{\rm alb}_{X},k}\ \ \text{ and }\ \ s \geq f_{{\rm alb}_{X\ast}(\cL^{\otimes f_{{\rm alb}_{X},k}}),k}\
		\]
		\item If, in addition, $\Nef(\Alb(X))$ is rational polyhedral then  $\cL^{\otimes rs}$ is globally generated in codimension $k$ for 
		\[
		r \dgeq f_{{\rm alb}_{X},k}\ \ \text{ and }\ \ s \dgeq 2\cdot \chi(X_{\alb_X},\cL_{\alb_X}^{\otimes f_{{\rm alb}_{X},k}})\ .
		\]
		where $X_{\alb_X}$ denotes the general fibre of the Albanese morphism of $X$ and $\cL_{\alb_{X}}$ is the restriction of $\cL$ to $X_{\alb{X}}$. 
	\end{enumerate}	
\end{thma}

Recall that the general polarized abelian variety has Picard rank 1; as such, (2) of Theorem A covers the case where ${\rm Alb}(X)$ is general in moduli.

The closed fibers of ${\rm alb}_{X}$ have numerically trivial canonical bundles themselves (cf.~ Proposition \ref{prop:fib-alb}), and the Fujita numbers of such varieties can be as low as 1 or at least as high as their dimension (see Examples \ref{ex:ab-var-f1}, \ref{ex:c-y} and \ref{ex:sm-hyp}).  Fibrations over abelian varieties have recently been studied by Cao--P\u{a}un \cite{CP}, who proved subadditivity of Kodaira dimension in this case;  their proof was recently simplified by Hacon--Popa--Schnell \cite{HPS}.    

The usefulness of Theorem~A in concrete situations depends largely on our ability to control the Fujita numbers of the bundles ${\rm alb}_{X\ast}\cL$. Below we will exhibit a couple of moduli spaces where this works out quite well. 

Combining our work with the aforementioned results of \cite{Riess} we obtain an effective global generation result for moduli spaces of stable sheaves on abelian surfaces.  We refer to Section \ref{sec:stable-sheaves} for the associated terminology and notation. 

\begin{thmb}[Theorem~\ref{thm:mod-sp}]
	Let $S$ be an abelian surface, let $v$ be a Mukai vector satisfying ${\langle}v^{2}{\rangle} \geq 6,$ and let $H$ be a very general ample divisor on $S$.  Then for $0 \leq k \leq {\rm dim}~{\cM_{H}(v)}$ we have 
	\begin{equation*}
	f_{\cM_{H}(v),k} \leq \max\biggl\{f_{K_{H}(v),k},\frac{{\langle}v^{2}{\rangle}}{2}+1\biggr\}.
	\end{equation*}
\end{thmb}

\noindent

For any ample line bundle $H$ on $S$ we have that $\cM_{H}((1,0,-n),\cO_{S})$ is isomorphic to the Hilbert scheme $S^{[n]}$ parametrizing length-$n$ subschemes of $S$, so Theorem B immediately implies

\begin{corc}[Corollaries~\ref{cor:yoshioka-mov}, \ref{cor:ab-surf-hilb}]
	Let $S$ be an abelian surface, let $n \geq 3$ be an integer, and let $S^{[n]}$ be the Hilbert scheme of length-$n$ subschemes of $S.$  
	\begin{itemize}
		\item[(1)]{$f_{S^{[n]},1} \leq n+1.$}
		
		\item[(2)]{If $S$ is an abelian surface and $\cL$ is an ample line bundle on $S^{[2]},$ then $\cL^{\otimes k}$ is globally generated for all $k \geq 3,$ i.e.~ $f_{S^{[2]}} \leq 3.$}
	\end{itemize}
\end{corc}

Note that (2) improves upon the bound $f_{S^{[2]}} \leq 5$ implied by \cite{Kaw}.

Returning to general considerations of Fujita numbers, in rough analogy with  the subadditivity of Kodaira dimension, several classes of examples (e.g.~ surfaces, products of varieties with no non-trivial correspondences) together with Theorem A point towards the following

\begin{conj*}
	Let $X,Y$ be  smooth projective varieties and   $\pi\colon X\to Y$ a smooth fibration. Then 
	\[
	f_X \dleq f_\pi\cdot f_Y\ 
	\] 
\end{conj*}

The organization of the paper can be summarized as follows:  Section \ref{sec:min-var} deals with the results we need about varieties with numerically trivial canonical bundle.  Section \ref{sec:sh} is devoted to the study of semihomogeneous bundles   Sections \ref{sec:thm-a} and \ref{sec:stable-sheaves} contain the proofs of Theorems A and B, respectively, while Section \ref{sec:ex-comp} contains various examples and some discussion around Fujita numbers. 

\subsection*{Acknowledgements} 
We would like to thank Thomas Bauer, Daniel Greb, Lawrence Ein, Stefan Kebekus, Sarah Kitchen, Vladimir Lazi\'c, Thomas Mettler, John C. Ottem, Mihai P\v{a}un, Mihnea Popa, Ulrike Riess, and Julius Ross for useful discussions related to this work.   The second author was supported by the Max-Planck-Institut f\"{u}r Mathematik while a substantial part of this work was done; he would like to thank them for their hospitality and excellent working conditions. 

\section{Minimal Varieties of Kodaira Dimension Zero}  
	\label{sec:min-var}

In this section we lay out the structure of our argument up to the point where we can use semihomogeneous vector bundles on complex tori.  We will use the following fundamental decomposition result for compact K\"ahler manifolds with numerically trivial canonical class; we refer to \cite{Bea} for the  proof.  Recall that a compact K\"{a}hler manifold $Y$ is \textit{hyperk\"{a}hler} if it is simply connected and $H^{0}(\Omega^{2}_{Y})$ is spanned by a symplectic 2-form (the term \textit{irreducible symplectic} is sometimes used) and \textit{Calabi-Yau} if $\dim(Y) \geq 3,$ the canonical bundle of $Y$ is trivial, and $H^{0}(\Omega^{p}_{Y})=0$ for $0 < p < \dim(Y).$  

\begin{theorem}[Beauville--Bogomolov decomposition]
	\label{thm:bb-dec}
	Let $X$ be a compact K\"ahler manifold with $c_1(X)=0$. Then there exists a finite \'etale cover ${\nu}' \colon \Pi_{i=1}^{d} X_i \to X$ such that every factor $X_i$ is a compact complex torus, a compact hyperk\"{a}hler manifold, or a Calabi--Yau manifold. 
\end{theorem}

\begin{prop}
	\label{prop:fib-alb}
	Let $X$ be a smooth projective variety satisfying $K_{X} =_{\rm num} 0$.  Then the Albanese fibration ${\rm alb}_{X} : X \to {\rm Alb}(X)$ fits into a commutative diagram
	\begin{equation}
	\label{eq:bb-diag}
	\xymatrix{
		{Y \times A' \times A''} \ar[rr]^-{\nu} \ar[d]^-{{\rm pr}_{A' \times A''}}  	& 					&X \ar[d]_-{{\rm alb}_{X}}  \\ 
		{A' \times A''} \ar[r]^-{{\rm pr}_{A'}}            						&A' \ar[r]^-{\rho}		&{\rm Alb}(X) 
	}
	\end{equation}
	where $Y$ is a smooth projective variety satisfying $K_{X} =_{\rm num} 0$ and $q(Y)=0$, $A'$ and $A''$ are abelian varieties, $\nu : Y \times A' \times A'' \to X$ is a surjective \'{e}tale map, and $\rho : A' \to {\rm Alb}(X)$ is an isogeny.  
\end{prop}

\begin{proof}
	By Theorem \ref{thm:bb-dec} there exists an \'{e}tale covering ${\nu}' : Y \times B \to X$ where $Y$ satisfies the claimed properties and $B$ is an abelian variety.  Since $q(Y)=0,$ we have that ${\rm Alb}(Y \times B) = B,$ and the Albanese map of $Y \times B$ is the projection ${\rm pr}_{B} : Y \times B \to B.$ 
	
	The Stein factorization of ${\rm alb}({\nu}') : B \to {\rm Alb}(X)$ has the form ${\rho}' \circ {\gamma},$ where ${\gamma} : B \to A'$ is an epimorphism of abelian varieties with connected kernel and ${\rho}' : A' \to {\rm Alb}(X)$ is an isogeny.  If $A'' := {\rm ker}(\gamma)$ then we have a diagram
	\begin{equation*}
	\xymatrix{
		{A' \times A''} \ar[r]^-{\theta} \ar[d]_-{{\rm pr}_{A'}}			&B \ar[d]_-{\gamma} 	\\
		{A'} \ar[r]^-{{\rho}''}    							&{A'}		
	}
	\end{equation*}
	where $\theta$ and ${\rho}''$ are isogenies.  Letting $\rho = {\rho}' \circ {\rho}''$ and $\nu = {\nu}' \circ (1_{Y} \times \theta),$ we see that the desired diagram (\ref{eq:bb-diag}) exists.
\end{proof}
 
\begin{rem}
	\label{rem:hyp}
	If $Y$ is a single point, then $X$ is a hyperelliptic variety in the sense of \cite{Lan}, and Theorem 1.2 of \cite{CI} implies that $\cL^{\otimes 2}$ is globally generated for any ample $\cL$ on $X,$ i.e.~ that $f_{X} \leq 2.$
\end{rem}

\section{Semihomogeneous Bundles}\label{sec:sh}

\subsection{Algebraic Preliminaries} 
We summarize the necessary algebraic properties of semihomogeneous bundles on abelian varieties. The Chern class computation is very likely to be known as well, but we did not find a concrete reference for it. Let $A$ be a complex abelian variety.

\begin{defn}
	A vector bundle $\cE$ on $A$ is \emph{semihomogeneous}, if for every $a\in A$ there exists a line bundle $\cL$ such that 
\begin{equation}
	\label{eq:sh-def}
	t_a^*\cE \simeq \cE\otimes \cL
\end{equation}

\noindent
where $t_a \colon A\to A$ stands for translation by $a$.  If $\cL \cong \cO_{A}$ in (\ref{eq:sh-def}) we say that $\cE$ is \textit{homogeneous.}
\end{defn}

Note that in this definition we must necessarily have $\cL\in \Pic^\circ(A)$.  We now collect some fundamental facts about semihomogeneous bundles from \cite{Muk}

\begin{prop}[Propositions 6.17 and 6.18, \cite{Muk}]\label{prop:semi-decomp} Let $\cE$ be a semihomogeneous bundle on $A.$  Then 
	\begin{enumerate}
		\item		for any semihomogeneous bundle $\cF$ on $A$, we have that $\cE \otimes \cF$ is semihomogeneous.
		\item 	there exist simple semihomogeneous bundles $\cE_{1}, \cdots \cE_{s}$ and indecomposable unipotent bundles $\cU_{1}, \cdots ,\cU_{s}$ on $A$ such that
		\begin{equation*}
		\label{eq:sh-decomp}
		\cE \cong \bigoplus_{j=1}^{s}\cE_{j} \otimes \cU_{j}\ .
		\end{equation*}
				Moreover, for each $j \in \{1, \cdots ,s\}$ there exists $\alpha_j \in \widehat{A}$ such that $\cE_{j} \cong \cE_{1} \otimes \alpha_{j}.$
	\end{enumerate}	
\end{prop}

\begin{prop}[Propositions 6.13 and 6.16, \cite{Muk}] 
	\label{prop:sh-stab}
	If $\cE$ is a semihomogeneous bundle on $A,$ and $H$ is an ample divisor on $A,$ then $\cE$ is $\mu_{H}$-semistable, and is $\mu_{H}$-stable if and only if $\cE$ is simple.
\end{prop}

\begin{prop}[Theorem 5.8 and Proposition 7.3, \cite{Muk}]\label{prop:simple sh}
	\label{prop:semihom-dir}
	Let $\cE$ be a simple vector bundle on $A.$  Then the following are equivalent. 
	\begin{enumerate}
		\item $\cE$ is semihomogeneous.
		\item There exists an abelian variety $A',$ an isogeny $p : A' \to A$ and a line bundle $\cL$ on $A'$ such that $\cE \cong p_{\ast}\cL.$
		\item There exists an abelian variety $B$ along with an isogeny $\sigma\colon B\to A$ and a line bundle $\cM$ on $B$ such that $\sigma^\ast \cE\cong \cM^{\oplus r}$, where $r=\rk \cE$. 
	\end{enumerate}
 \end{prop}

\begin{prop}[Proposition 5.4, \cite{Muk}]
	\label{prop:semi-isog}
	Let $A'$ and $A''$ be abelian varieties, and let ${\pi}' : A' \to A$ and ${\pi}'' : A'' \to A'$ be isogenies.  If $\cE$ is a vector bundle on $A',$ then all three of the vector bundles $\cE, {\pi}'_{\ast}\cE, {\pi}''^{\ast}\cE$ are semihomogeneous if and only if one of them is semihomogeneous.
\end{prop}

\begin{lem}[Chern classes of semihomogeneous bundles]\label{lem:Chern classes of sh}
Let $\cE$ be a semihomogeneous vector bundle on the abelian variety $A$. Then it has the total Chern class of a direct sum, that is,
\begin{equation*}
	\label{eq:sh-chern}
	c(\cE) \equ {\left( 1 + \frac{c_1(\cE)}{r}\right)}^r
\end{equation*}
with $r=\rk \cE$.
\end{lem}
\begin{proof}
Assume first that $\cE$ is simple. Then Proposition~\ref{prop:simple sh} yields the existence of an isogeny $\sigma\colon B\to A$ such that $\sigma^\ast\cE\cong \cM^{\oplus \rk \cE}$ for a line bundle $\cM$ on $B$. Then 
\[
c(\sigma^\ast\cE) \equ c(\cM^{\oplus r}) \equ {(1+c_1(\cM))}^r \equ {\left(1 + \frac{c_1(\sigma^*\cE)}{r} \right)}^r\ .
\]
By the projection formula for finite flat morphisms, we then have 
\[
\deg{\sigma} \cdot c(\cE) = \sigma_{\ast}c(\sigma^{\ast}\cE) = \sigma_{\ast}{\left(1 + \frac{c_1(\sigma^*\cE)}{r} \right)}^r\ = \deg{\sigma} \cdot {\left(1 + \frac{c_1(\cE)}{r} \right)}^r\
\]
from which (\ref{eq:sh-chern}) follows at once for simple semihomogeneous bundles.

Dropping the simplicity assumption, let $\cE$ now be an arbitrary semihomogeneous bundle. By Proposition~\ref{prop:semi-decomp}, there exists a \ssh bundle $\cE_1$, topologically trivial line bundles $\alpha_j$, and unipotent bundles $\cU_j$ such that 
\[
	\cE \cong \bigoplus_{j=1}^{s}\cE_{1}\otimes\alpha_j \otimes \cU_{j}\ .
\] 
Since $c(\cU_j) = c(\alpha_{j})=1$ for all $j,$ we have from the previously established case of simple bundles that 
\[
c(\cE) \equ \prod_{j=1}^{s} c(\cE_{1}\otimes\alpha_j \otimes \cU_{j}) \equ \prod_{j=1}^{s} c(\cE_1)^{\rk \cU_j}  \equ \prod_{j=1}^{s} \left(1+ \frac{c_1(\cE_1)}{\rk \cE_1} \right)^{\rk \cU_j}
\equ \left(1 +\frac{c(\cE_1)}{\rk \cE_1}\right)^{\rk \cE}\ ,
\]
and by $\tfrac{c(\cE_1)}{\rk \cE_1}=\tfrac{c(\cE)}{r}$ the proof is complete. 
\end{proof}

\begin{prop}
\label{prop:semi-dir-im}
	Let $\phi : B \to A$ be a surjective homomorphism of abelian varieties, and let $E$ be an ample semihomogeneous bundle on $B$ which is simple.  Then $\phi_{\ast}E$ is an ample semihomogeneous bundle on $A.$
\end{prop}

\begin{proof}
	Throughout the proof we assume without loss of generality that $\phi$ does not factor through an isogeny.  First we prove $\phi_{\ast}E$ is ample.  For all $\alpha \in \widehat{A}$ and all $i > 0, 0 \leq j \leq i$ we have

\begin{equation*}
	H^{j}(R^{i-j}\phi_{\ast}(E \otimes \phi^{\ast}\alpha)) \Rightarrow H^{i}(E \otimes \alpha)
\end{equation*}

We have that $R^{i-j}\phi_{\ast}(E \otimes \phi^{\ast}\alpha) = 0$ for $i > j,$ so when $i > 0$ this forces 
	\begin{equation*}
		H^{i}(\phi_{\ast}E \otimes \alpha) \cong H^{i}(E \otimes \phi^{\ast}\alpha) \cong 0
	\end{equation*}
It follows that $\phi_{\ast}E$ is I.T. of index 0, and is therefore an ample vector bundle as claimed.  In particular, $\phi_{\ast}E$ is nondegenerate.

	To prove semihomogeneity, it suffices by the nondegeneracy of $\phi_{\ast}E$ to show that the correspondence
	\begin{equation*}
		\{(x,\alpha) \in A \times \widehat{A}~:~ t_{x}^{\ast}\phi_{\ast}E \cong \alpha \otimes \phi_{\ast}E\}
	\end{equation*}
	has dimension ${\rm dim}(A).$  For this it suffices in turn to show that for all $\alpha \in \widehat{A}$ there exists $x \in A$ for which $\alpha \otimes \phi^{\ast}E \cong t_{x}^{\ast}\phi_{\ast}E.$  
	
	Let $\alpha \in \widehat{A}$ be given.  Since $E$ is an ample semihomogeneous bundle on $B$ which is simple, the correspondence 
	\begin{equation*}
		\{(y,\beta) \in B \times \widehat{B}~:~ t_{y}^{\ast}E \cong \beta \otimes E\}
	\end{equation*}
	has dimension ${\rm dim}(B).$  Thus there exists $y' \in B$ such that $\phi^{\ast}\alpha \otimes E \cong t_{y'}^{\ast}E.$  Setting $x := \phi(y')$ we then have
	\begin{equation*}
		\alpha \otimes \phi_{\ast}E \cong \phi_{\ast}(\phi^{\ast}\alpha \otimes E) \cong \phi_{\ast}(t_{y'}^{\ast}E) \cong t_{\phi(y')}^{\ast}\phi_{\ast}E = t_{x}^{\ast}\phi_{\ast}E.
	\end{equation*}
\end{proof}

\subsection{Effective global generation of ample semihomogeneous bundles}

Here we analyze the possibilities for global generation of ample semihomogeneous vector bundles. First, relying on a result of Pareschi, we provide an effective upper bound in terms of the rank of the given vector bundles and the geometry of the underlying abelian variety. It turns out that one has reasonably good control whenever the nef cone of the abelian variety in question is rational polyhedral.  In the other direction  we point out using a class of examples from \cite{Op} that there is no uniform bound for effective global generation of ample semihomogeneous bundles. 

\begin{rem}
	\label{rem:gg-prim}
	Every abelian variety $A$ admits an ample semihomogeneous bundle which is not globally generated.  Indeed, let $\rho : A' \to A$ be an isogeny for which $A'$ admits a principal polarization $\Theta,$ and note that the ample semihomogeneous bundle $\rho_{\ast}\Theta$ fails to be globally generated (cf. Example~\ref{ex:ab-var-f1}).  
\end{rem} 

\begin{prop}
	\label{prop:pareschi-gg}
	Let $\cF$ be a semihomogeneous vector bundle on an abelian variety $A,$ and let $\cM$ be an ample line bundle on $A$ such that $c_{1}(\cF \otimes \cM^{-1})$ is ample.  Then $\cF$ is globally generated.
\end{prop}

\begin{proof}
	Since $\cF$ is semihomogeneous, we have from Proposition 2.6 in \cite{KM18} that $\cF \otimes \cM^{-1}$ is I.T. of index 0, i.e.~ that $H^{i}(\cF \otimes \cM^{-1} \otimes \alpha) = 0$ for all $i > 0$ and all $\alpha \in {\rm Pic}^{0}{A}.$  The global generation of $\cF$ then follows from Theorem 2.1 of \cite{Pa}.
\end{proof}

\begin{rem}
	While the conclusion of Theorem 2.1 of \cite{Pa} holds under the weaker hypothesis that $\cF$ is M-regular \cite{PP1}, the latter property is equivalent to being I.T. of index 0 when $\cF$ is semihomogeneous by Proposition 2.6 of \cite{KM18}.
\end{rem}

We introduce an invariant of projective varieties that will help formulate  an effective bound on global generation of ample semihomogeneous bundles.  
 
\begin{defn}[Ample slope]
Let $X$ be a projective variety, $M$ an ample divisor on $X$. We define the \emph{ample slope of $X$ with respect $M$} as 
\[
\mu^{\textrm{amp}}_M(X) \deq \inf\,\{ m\in\N\,|\,\text{$mL-M$ is ample for every ample Cartier divisor $L$ on $X$}\}\ ,
\]
and the \emph{ample slope of $X$} as
\[
\mu^{\textrm{amp}}(X) \deq \inf\, \{\mu^{\textrm{amp}}_m(X)\,|\, \text{$M$ is ample on $X$}\}\ .
\]	
\end{defn}
 
\begin{rmk}
Observe that $\mu^{\textrm{amp}}(X) \geq 2$ for all $X$, at the same time, if $\rho(X)=1$ then  $\mu^{\textrm{amp}}(A)=2$, since $\mu^{\textrm{amp}}_M(X)=2$  when $M$ is taken to be an ample generator of the N\'eron--Severi space. Note that  $\mu^{\textrm{amp}}(X)$ need not be finite.  
\end{rmk}

We are interested in the ample slope of an abelian variety. For starters here we give an example of an abelian surface with infinite ample slope. 

\begin{eg}
	We capitalize on  an example of  Koll\'ar \cite{PAGI}*{Example 1.5.7}). Let $A=E\times E$ be the self-product of an elliptic surface. Using the notation of \emph{loc. cit.} we consider the sequence of divisors
	\[
	A_n \equ nF_1 + (n^2-n_1)F_2 - (n-1)\Delta\ . 
	\]
	Then $(A_n^2)=2$. Given an arbitrary ample divisor $M$ on $A$, we see that
	\[
	(A_n\cdot M) \equ (F_2\cdot M)n^2 + ((F_1\cdot M) - (F_2\cdot M) - (\Delta\cdot M))n + ((f_2\cdot M) + (\Delta\cdot M))\ .  
	\]
	Therefore, 
	\[
	((nA_n - M)^2) \equ 2n^2 -2n(A_n\cdot M) + (M^2)  \equ -(F_2\cdot M)n^3 + \text{lower order terms in $n$}     ,
	\]
	which will become negative once $n$ is large enough since $(F_2\cdot M)\neq 0$ by virtue of the ampleness of $M$. 
\end{eg}

\begin{prop}
	\label{prop:rat-poly-nef}
	If $X$ is an abelian variety whose nef cone is rational polyhedral, then 
	\[
	\mu^{\rm amp}(X) \equ 2\ .
	\]
\end{prop}

\begin{proof}
	By the main theorem of \cite{Bau} there is an isogeny 
	\[
	\sigma \colon Y \deq  \Pi_{j=1}^{l}Y_{j} \to X\ ,
	\]
	where $Y_{1}, \cdots , Y_{l}$ are pairwise non-isogenous abelian varieties with Picard number one. It is easy to see that all $Y_i$'s have  ample slope equal to two.  First we claim that $\mu^{\rm amp}(Y) = 2;$ to prove this, it is enough to show that $\mu^{\rm amp}(Y) \leq 2.$  
	
	For each $1\leq j\leq l$ let $\cM_j$ be an ample generator of ${\rm NS}(Y_{j})$.  If $\cL$ is any ample line bundle on $Y,$ then we have $\cL \cong \cL_{1} \boxtimes \ldots \boxtimes \cL_{l}$ where $\cL_{j}$ is an ample line bundle on $Y_{j}$.  It follows at once that 
	\[
	\cL^{\otimes 2}\otimes \left(\cM_{1} \boxtimes \ldots \boxtimes \cM_{l}\right)^{-1}
	\] 
	is ample, so that 
	\[
	\mu^{\rm amp}(Y) \leq \mu^{\rm amp}_{M_{1} \boxtimes \cdots \boxtimes M_{\ell}}(Y) \leq 2\ ,
	\]
	which proves our  claim.
	
	Next, we will compare $\mu^{\rm amp}(X)$ to $\mu^{\rm amp}(Y)$. To this end,  let $\cM'$ be an ample line bundle on $Y$ for which 
	\[
	\mu_{\cM'}^{\rm amp}(Y) \equ  \mu^{\rm amp}(Y) \equ 2\ ,
	\] 
	and let $\cL$ be an ample line bundle on $X.$  Since ${\rm NS}(X)$ is a finite-index subgroup of ${\rm NS}(Y),$ we have that ${\cM'}^{\otimes d} = \sigma^{\ast}\cM$ for an ample line bundle $\cM$ on $X.$  By assumption, the line bundle 
	\[
	\left(\sigma^{\ast}\cL\right)^{\otimes 2d} \otimes {\cM'}^{\otimes -d}
	\]
	is ample,  and this implies at once that $\cL^{\otimes 2}\otimes \cM^{\otimes -1}$ is ample.  It follows that
	\begin{equation*}
	\mu^{\rm amp}(X) \leq \mu^{\rm amp}_{M}(X) \leq 2.
	\end{equation*}
\end{proof}

\begin{prop}\label{prop:ash gg if finite ample slope}
	Let $\cE$ be an ample semihomogeneous vector bundle on an abelian variety $A$ with finite ample slope. If
	\[
	m \dgeq \rk \cE \cdot \mu^{\textrm{amp}}(A)\ 
	\]
	then $\Sym^m\cE$ is globally generated.
\end{prop}

\begin{proof}
	We will prove that $\cE^{\otimes m}$ is globally generated in the appropriate range, this then implies the global generation of $\Sym^m\cE$ as well. Let $\cL$ be an ample line bundle on $A$ realizing the minimal ample slope.  By Proposition \ref{prop:pareschi-gg}, it is enough to show that $c_{1}({\cE}^{\otimes m} \otimes {\cL}^{-1})$ is ample.  We compute
	\[
	c_1(\cE^{\otimes m}\otimes \cL^{-1}) \equ m\cdot c_1(\cE) - \rk \cE\cdot c_1(\cL)\,
	\]
	which is ample by the definition of ample slope if $m\geq \rk\cE \cdot \mu^{\textrm{amp}}_M(X)  \equ \rk \cE \cdot \mu^{\textrm{amp}}(X)$, as required.  
\end{proof}

\begin{thm}\label{thm:Fujita bound for ash on rtl poly}
Let $A$ be an abelian variety with a finite rational polyhedral nef cone, $\cE$ an ample semihomogeneous vector bundle on $A$. Then  
\[
f_\cE \dleq 2{\rm rk}\cE\ . 
\] 
\end{thm}
\begin{proof}
The statement is the conjunction of  Propositions~\ref{prop:rat-poly-nef} and \ref{prop:ash gg if finite ample slope}. 
\end{proof}

\begin{rmk}
The argument of the proof of Proposition~\ref{prop:ash gg if finite ample slope} yields a numerical condition of somewhat different flavour under the stronger assumption that $A$ has Picard number 1.  Let $\cL$ be an ample line bundle $\cL$ generating  the N\'eron--Severi group of $A$. With this assumption one can prove that an ample semihomogeneous vector bundle $\cE$ on $A$ is globally generated if 
\[
\mu_{\cL}(\cE) > c_{1}(\cL)^{g}\ ,
\]
where $g$ stands for the dimension of $A$. 

Again, one proceeds via Proposition \ref{prop:pareschi-gg} by verifying that $c_1(\cE\otimes \cL^{-1})$ is ample. Our hypothesis on the N\'eron--Severi group 	
of $A$ implies that $c_{1}(\cE \otimes {\cL}^{-1})=sc_{1}(\cL)$ for some $s \in \Z$; the remaining task is then to show that $s > 0$. On the one hand, 
\[
\mu_{\cL}(\cE \otimes \cL^{-1}) \equ \mu_{\cL}(\cE) -  c_1(\cL^{-1})^g > 0
\]
by assumption, on the other hand 
\[
\mu_{\cL}(\cE \otimes \cL^{-1}) \equ \frac{sc_{1}(\cL)^{g}}{\rk~{\cE})}-c_{1}(\cL)^{g} \equ (s/\rk\cE -1)\cdot c_1(\cL)^g \ ,
\]
and the statement follows. 
\end{rmk}

\begin{rem}
	By Theorem 1.1 of \cite{BaSz} it follows that $\cL_1\otimes\cL_2$ is globally generated provided $\cL_1,\cL_2$ are ample line bundles on an abelian variety. Consequently, if $\cE$ is a direct sum of ample line bundles on an abelian variety then $\cE^{\otimes 2}$ hence also $\Sym^2\cE$ are globally generated.  
\end{rem}

\begin{rem}
	\label{rem:zhang}
	Example 3.4 of \cite{PP2} exhibits vector bundles $E$ and $F$ on a principally polarized abelian variety $(A,\Theta)$ such that $E \otimes F$ is not globally generated and $E,F$ are both I.T. of index 0; it follows that ${\rm Sym}^{2}(E \oplus F)$---which contains $E \otimes F$ as a direct summand---is not globally generated, although $E \oplus F$ is I.T. of index 0.  Since $E$ and $F$ can be shown to have different slopes with respect to $\Theta,$ their direct sum $E \oplus F$ is not slope-semistable with respect to $\Theta,$ and is therefore not semihomogeneous by Proposition \ref{prop:sh-stab}.
\end{rem}

The following observation shows that although $\Sym^2\cE$ might not be globally generated for an ample semihomogeneous vector bundle, a suitable pullback will always be so. 

\begin{prop} \label{prop:pullback-tensor-square}
	Let $\cE$ be an ample semihomogeneous bundle on $A$, and let $m \geq 2$ be an integer.  Then there exists an isogeny $\sigma\colon A'\to A$ such that $\sigma^{\ast}\cE^{\otimes m}$ is globally generated.  In particular, $\sigma^{\ast}\Sym^{m}(\cE)$ is globally generated.
\end{prop}

\begin{proof}
	By Proposition \ref{prop:semi-decomp}, the indecomposable summands of $\cE^{\otimes m}$ are all of the form $\cU' \otimes \cE_{j_1} \otimes \cdots \otimes \cE_{j_m}$ for $\cU'$ unipotent and $\cE_{j_1}, \cdots ,\cE_{j_m}$ simple and semihomogeneous.  Given that an extension of ample globally generated semihomogeneous bundles on an abelian variety is globally generated, we are reduced to considering $\cE_{j_1} \otimes \cdots \otimes \cE_{j_m}.$  By Proposition~\ref{prop:semi-decomp} (ii),  we have  $\cE_{j_k} \cong \cE_{j_1} \otimes \alpha_k$ for some $\alpha_k \in \widehat{A}$ for all $2 \leq k \leq m$.  If ${\alpha}'$ is any $m$-th root of $\alpha_{2} \otimes \cdots \otimes \alpha_{m}$, then 
	\[
	\cE_{j_1} \otimes \cdots \otimes \cE_{j_m} \cong (\cE_{j_1} \otimes {\alpha}')^{\otimes m} \ .
	\]
	We may therefore assume without loss of generality that $\cE$ is simple.
	
	Let $\sigma : A' \to A$ be an isogeny and let $\cL$ be an ample line bundle on $A'$ such that $\cE \cong \sigma_{\ast}\cL.$  We have that
	\begin{equation*}
	\sigma^{\ast}(\cE^{\otimes m}) \cong (\sigma^{\ast}\sigma_{\ast}\cL)^{\otimes m} \cong (\oplus_{x \in {\rm ker}(\sigma)}t_{x}^{\ast}\cL)^{\otimes m}
	\end{equation*}
	It then follows from the standard result for ample line bundles on abelian varieties that $\sigma^{\ast}(\cE^{\otimes m})$ is globally generated.  
\end{proof}

We end this subsection with an example showing that the global generation can fail for low tensor powers of ample semihomogeneous bundles.  Following the notation in \cite{Op}, let $(A,\Theta)$ be a principally polarized abelian variety of dimension $g \geq 1$ with $\Theta$ symmetric (i.e.~ $(-1)^{\ast}\Theta \cong \Theta$) and let $a > 1$ be an odd integer; there exists a simple semihomogeneous vector bundle $\mathbf{W}_{a,1}$ on $A$ satisfying the following:
\begin{equation}
\label{eq:rk-det}
{\rm rk}(\mathbf{W}_{a,1}) = a^{g} \hskip50pt \det{\mathbf{W}_{a,1}} \cong a^{g-1}\Theta .
\end{equation}
We also have (see p.~ 14 of \cite{Op}) that
\begin{equation}
\label{eq:ch-w}
{\rm ch}(\mathbf{W}_{a,1}) = a^{g} \cdot {\rm exp}\biggl(\frac{1}{a}\Theta\biggr)\ .
\end{equation}
The vector bundles $\mathbf{W}_{a,1}$ are identical to the Verlinde bundles $E_{a,1}$ defined by Popa \cite{Popa} (for other integers $r$ the analogous bundles $\mathbf{W}_{a,r}$ occur as indecomposable factors of the Verlinde bundles $E_{a,r}$). 

\begin{prop}\label{prop:counterexample}
	$\mathbf{W}_{a,1}^{\otimes k}$ fails to be globally generated for all odd $k\leq a$.   
\end{prop}

\begin{proof}
	Since $\mathbf{W}_{a,1}$ is semihomogeneous with ample determinant, it is ample by Proposition 2.6 of \cite{KM18} and therefore it suffices to show that $h^{0}(\mathbf{W}_{a,1}^{\otimes k}) <  \rk(\mathbf{W}_{a,1}^{\otimes k})+\dim(A)$ for all odd $k \leq a.$  We fix such a $k$ for the remainder of the proof.
   
	By another application of \textit{loc.~ cit.} the ampleness and semihomogeneity of $\mathbf{W}_{a,1}$ imply that all its tensor powers are I.T. of index 0.  This, combined with (\ref{eq:ch-w}), gives 
	\begin{equation*}
	h^{0}(\mathbf{W}_{a,1}^{\otimes k}) \equ \chi(\mathbf{W}_{a,1}) \equ {\rm ch}_{g}(\mathbf{W}_{a,1}^{\otimes k}) \equ  a^{kg} \cdot \biggl[{\rm exp}\biggl(\frac{k}{a}\Theta\biggr)\biggr]_{g}
   \equ (ka^{k-1})^{g} \ .
	\end{equation*}
	At the same time it follows from (\ref{eq:rk-det}) that 
	\begin{equation*}
	{\rm rk}(\mathbf{W}_{a,1}^{\otimes k})+\dim(A) = a^{kg}+g
	\end{equation*}
	The desired result then follows from insisting on the inequality 
	\[
	(ka^{k-1})^{g} < (a^{k})^{g} + g\ ,  
	\]	
	or, equivalently,
    \begin{equation*}
    \biggl(\frac{k}{a}\biggr)^{g} < 1 + \frac{g}{a^{kg}}\ .
    \end{equation*}	
	which certainly holds under our hypothesis.
\end{proof}

\begin{rem}
Proposition~\ref{prop:counterexample} exhibits an ample semihomogeneous bundle $\cE$ on an abelian variety of dimension $g$ for which $f_{\cE} > ({\rm rk}~{\cE})^{1/g}.$  However, it is still open whether there exists a constant $\gamma$ depending only on the underlying abelian variety $A$ such that 
\[
f_\cE \leq \gamma~ \rk \cE
\]
for all ample semihomogeneous bundles $\cE$ on $A$. In Theorem~\ref{thm:Fujita bound for ash on rtl poly} we saw that one can take $\gamma=2$ when $A$ has a rational polyhedral nef cone. 
\end{rem}

\section{Proof of Theorem A}
	\label{sec:thm-a}

Before proceeding to the proof of our main result, we make a brief detour to discuss the use of fibrations in ensuring global generation. The statements below are surely known to experts; we include them for the sake of completeness and the lack of a precise reference. 
 
\begin{lem}
	\label{lem:adj-surj}
	Let $f\colon Y\to Z$ be a morphism of projective varieties, $\cE$ a vector  bundle on $Y$. Assume furthermore that the adjunction map $f^\ast f_\ast\cE\to\cE$ is surjective. 
	\begin{enumerate}
		\item If $\Sym^m (f_\ast\cE)$ is globally generated for some positive integer $m$ then so is $\Sym^m\cE$.  
		\item If $\omega_Z\otimes f_\ast(\omega_{Y/Z}\otimes \cE)$ is globally generated then so is $\omega_Y\otimes \cE$. 
	\end{enumerate}
\end{lem}
\begin{proof}
	For (1), the adjunction morphism is surjective, so the same is true of the induced map ${\Sym}^{m}(f^\ast f_\ast{\cE}) \to {\rm Sym}^{m}\cE.$  It follows that the composition 
	\[
	f^*({\rm Sym}^{m}f_\ast \cE) \,\simeq\, {\rm Sym}^{m}(f^\ast f_\ast{\cE}) {\longrightarrow} {\rm Sym}^{m}\cE
	\]
	is surjective as well. As  $f^*({\rm Sym}^{m}f_\ast \cE)$ is globally generated by assumption, we have that ${\rm Sym}^{m}{\cE}$ is globally generated as claimed.
	Part (2) follows by a similar argument using the projection formula 
	\[
	f_\ast(\omega_Y\otimes \cE) \cong  f_\ast(\omega_{Y/Z}\otimes f^\ast\omega_Z\otimes \cE) \cong  \omega_Z \otimes f_\ast(\omega_{Y/Z}\otimes \cE) \ ,
	\]
	and the surjectivity of the adjunction map. 
\end{proof}
 
\begin{rem}[Variant for continuous global generation]
	The above statements have natural counterparts in the setting of  continuous global generation (cf. \cite{PP1}). Here is an example. Let $Y$,$Z$ be projective varieties, $Z$ irregular, $f\colon Y\to Z$ be  an arbitrary morphism of varieties. Let $\cF$ be a coherent sheaf, $\cL$ a line bundle 
	on $Y$. Assume that 
	\begin{enumerate}
		\item $f_\ast\cF$ and $f_\ast \cL$ are continuously globally generated; 
		\item the adjunction morphisms $f^\ast f_\ast\cF\to \cF$ and $f^\ast f_\ast \cL\to \cL$ are surjective. 
	\end{enumerate}
	Then $\cF\otimes \cL$ is globally generated. 
	
	Indeed, this is a slight variation of \cite{PP1}*{Proposition 2.12}. Due to \emph{loc.~cit.} applied to the identity morphism of $Z$, the sheaf $f_\ast \cF\otimes f_\ast \cL$ is  globally generated. This implies that 
	\[
	f^\ast f_\ast\cF \otimes f^\ast f_\ast \cL \,\cong\, f^*(f_\ast\cF \otimes f_\ast \cL)
	\]
	is globally generated as well. By the surjectivity of the adjunction morphisms in (2), the tensor product 
	\[
	f^\ast f_\ast\cF \otimes f^\ast f_\ast \cL \stackrel{\otimes}{\longrightarrow} \cF\otimes \cL
	\]
	is surjective, too. Therefore $\cF\otimes \cL$ is globally generated as promised.  
\end{rem}

Recall the following definition. 

\begin{defn}
We say a line bundle $\cL$ on a projective variety $X$ is \emph{globally generated in codimension $k$} if the base locus ${\rm Bs}|\cL|$ is of codimension at least $k+1$ in $X.$	
\end{defn}

Next, we define Fujita numbers of ample sheaves, which leads to a quantified version of the circle of ideas around Fujita's conjecture. 

\begin{defn}[Fujita numbers]
Let $X$ be a smooth projective variety, $\cE$ an ample coherent sheaf on $X$, and let  $0\leq k\leq \dim X$ be a natural number. We define 
the \emph{$k$\textsuperscript{th} Fujita number of $\cE$} as 
\[
f_{\cE,k} \deq \min \{m\geq 1\,|\, \omega_X\otimes \Sym^{m'}\cE \text{ is  globally generated in codimension $k$ for all $m'\geq m$} \}\ .
\]
Along the same lines we define the Fujita number  $f_{X,k}$ of a given smooth projective variety $X$  as 
\[
f_{X,k} \deq \max\{ f_{\cA,k}\,|\, \text{$\cA$ is an ample line bundle on $X$}\}\ ,
\]
while for a smooth fibration $\pi\colon X\to Y$ we set $f_{\pi,k} \deq \max \st{f_{X_y,k}\colon y\in Y}$.
\end{defn}

\begin{rmk}
Note that the sequences $\{f_{\cE,k}\}_{k}, \{f_{X,k}\}_{k}, \{f_{\pi,k}\}_{k}$ are all nondecreasing.  Going forward we write $f_{\cE,\dim{X}}$ as $f_{\cE}$ (similarly for $f_{X,\dim{X}}$ and $f_{\pi,\dim{X}}$). 
\end{rmk}

\begin{lem}
	If $X$ is a smooth projective variety and $\cE$ is an ample vector bundle on $X,$ then $f_{\cE,k}$ is finite for $0 \leq k \leq \dim{X}.$
\end{lem}

\begin{proof}
	It suffices to check that $f_{\cE}$ is finite.  By (iii) of Theorem 6.1.10 in \cite{PAGII} there exists a positive integer $m$ such that $\omega_{X} \otimes {\rm Sym}^{m'}\cE$ is globally generated for all $m' \geq m$.  Therefore $f_{\cE} \leq m,$ so that $f_{\cE}$ is finite as claimed.
\end{proof}
 
\begin{rmk}
Given a fibration $\pi\colon X\to Y$ of varieties, we write $X_\pi$ for the general fibre of $\pi$. Analogously, if $\cL$ is a line bundle on $X$ then $\cL_\pi$ stands for $\cL|_{X_\pi}$. 
\end{rmk}

\begin{thm}\label{thm:codim of base loci}
	Let $X$ be a smooth projective variety with $K_{X} =_{\rm num} 0,$ and let $\cL$ be an ample line bundle on $X$.  Then for $1 \leq k \leq {\rm dim}~{X}$ we have that
	\begin{enumerate}
		\item $\cL^{\otimes rs}$ is globally generated in codimension $k$ for 
		\[
		r \geq f_{{\rm alb}_{X},k}\ \ \text{ and } \ s \geq f_{{\rm alb}_{X\ast}(\cL^{\otimes f_{\rm alb_{X},k}}),k}\
		\]
		\item If, in addition, $\Nef(\Alb(X))$ is rational polyhedral then  $\cL^{\otimes rs}$ is globally generated in codimension $k$ for 
		\[
		r \dgeq f_{{\rm alb}_{X},k}\ \ \text{ and }\ \ s \dgeq 2\cdot \chi(X_{\alb_X},\cL_{\alb_X}^{\otimes f_{\rm alb_{X},k}})\ .
		\]
	\end{enumerate}
\end{thm}

\begin{proof}
	 Consider the  diagram  (\ref{eq:bb-diag}) from Proposition~\ref{prop:fib-alb}  
	\begin{equation}\label{eq:diagram in the proof}
	\xymatrix{
		{Y \times A' \times A''} \ar[rr]^-{\nu} \ar[d]^-{{\rm pr}_{A' \times A''}}  	& 					&X \ar[d]_-{{\rm alb}_{X}} 	\\
		{A' \times A''} \ar[r]^-{{\rm pr}_{A'}}            						&A' \ar[r]^-{\rho}		&{\rm Alb}(X) \ .
	}
	\end{equation}	
	Define $\cF_1 \deq  (\pr_{A'\times A''})_\ast \nu^\ast\cL$.  Since $\nu^{\ast}\cL$ is ample and there are no non-trivial correspondences between $Y$ and $A'\times A''$, $\nu^\ast\cL=\cM_1\boxtimes\cM_2$ for ample line bundles $\cM_1$ on $Y$ and $\cM_2$ on $A'\times A''$.  
	
	We claim that $H^{0}(Y,\cM_{1}) \neq 0.$  If $\alpha : X' \to A'$ is the base change of ${\rm alb}_{X}$ via $\rho,$ then we have the factorization
	\begin{equation}\label{eq:diagram in the proof}
	\xymatrix{
		{Y \times A' \times A''} \ar[r]^-{\widetilde{\nu}} \ar[d]^-{{\rm pr}_{A' \times A''}} \ar@/^2pc/[rr]^-{\nu}  	&X' \ar[d]_-{\alpha} \ar[r]^-{\widetilde{w}} 		&X \ar[d]_-{{\rm alb}_{X}} 	\\
		{A' \times A''} \ar[r]^-{{\rm pr}_{A'}}            								&A' \ar[r]^-{\rho}						&{\rm Alb}(X) \ .
	}
	\end{equation}	
	induced by the universal property of the fiber product.  For all $a' \in A'$ we have that
	\begin{equation}
		H^{0}(X'_{a'},\widetilde{w}^{\ast}\cL|_{X'_{a'}}) \cong H^{0}(X_{\rho(a')},\cL_{\rho(a')})
	\end{equation}
	Together with this isomorphism, our hypothesis on the base locus of $\cL_{\rho(a')}$ implies that the base locus of $\widetilde{w}^{\ast}\cL|_{X'_{a'}}$ is of codimension $k$; in particular $H^{0}(X'_{a'},\widetilde{w}^{\ast}\cL|_{X'_{a'}}) \neq 0.$  We then have that
	\begin{equation}
		0 \neq H^{0}(X'_{a'},\widetilde{w}^{\ast}\cL|_{X'_{a'}})  \subset H^{0}(Y \times \{a'\} \times A'',\widetilde{\nu}^{\ast}\widetilde{w}^{\ast}\cL|_{Y \times \{a'\} \times A''})	
	\end{equation}
	\begin{equation*}
		 \cong H^{0}(Y \times \{a'\} \times A'',\nu^{\ast}\cL|_{Y \times \{a'\} \times A''}) \cong H^{0}(Y,\cM_{1}) \otimes H^{0}(\{a'\} \times A'', \cM_{2}|_{\{a'\} \times A''}) 
 	\end{equation*}
	It then follows that $H^{0}(Y,\cM_{1}) \neq 0$ as desired.
	
	Pushing forward via the projection to $A'\times A''$ yields
	\[
	\cF_{1} \cong (\pr_{A'\times A''})_\ast(\cM_1\boxtimes\cM_2) \cong ((\pr_{A'\times A''})_\ast(\pr_Y)^\ast\cM_1)\otimes \cM_2 \cong H^0(Y,\cM_1)\otimes \cM_2\ ,
	\]
	Since $\cM_2$ is ample and $H^{0}(Y,\cM_1) \neq 0,$ it follows that $\cF_2\deq (\pr_{A'})_\ast\cF_1$ is an ample semihomogeneous bundle on $A'$ thanks to Proposition \ref{prop:semi-dir-im}.  The morphism $\rho\colon A'\to\Alb(X)$ is an isogeny, so $\rho_\ast\cF_2$ is an ample semihomogeneous vector bundle on $\Alb(X)$. By the commutativity of the diagram (\ref{eq:diagram in the proof}), 
	\[
	\rho_\ast\cF_2 \cong (\alb_X)_\ast(\nu_\ast\nu^\ast\cL) \cong (\alb_X)_\ast(\cL\otimes \nu_\ast\cO_{Y\times A'\times A''})\ .
	\]
	As $\nu$ is finite and \'etale, $\nu_\ast\cO_{Y\times A'\times A''}\cong \cO_X\oplus \cG$ for a vector bundle $\cG$ on $X$, therefore 
	\[
	\rho_\ast\cF_2 \cong (\alb_X)_\ast\cL \oplus (\alb_X)_\ast(\cL\otimes \cG)\ .
	\]
	Being a direct summand of an ample semihomogeneous vector bundle, $(\alb_X)_\ast\cL$ is ample and semihomogeneous itself.   
	
	In order to prove (1), observe that our assumption on the base loci of the $\cL_{a}$ implies that the map $(\alb_X)^\ast(\alb_X)_\ast\cL \to \cL$ is surjective away from a subset of codimension $k$, so the same is true of its symmetric power 
	\[
	\Sym^s ({\alb_X}^\ast{\alb_X}_\ast(\cL^{\otimes r})) \to \cL^{\otimes rs}
	\]
	It follows that $\cL^{\otimes rs}$ is globally generated in codimension $k$ as desired, which yields (1). 
	
	To show statement (2), we invoke Theorem~\ref{thm:Fujita bound for ash on rtl poly}, which implies that $\Sym^{s} \left({\alb_X}_\ast (\cL^{\otimes r})\right)$ is globally generated provided $s \geq 2\cdot  \rk {\alb_X}_\ast(\cL^{\otimes r})$.   
	We have 
	\[
		\rk {\alb_X}_\ast(\cL^{\otimes r}) \equ h^0(X_{\alb_X},\cL_{\alb_X}^{\otimes r}) \equ \chi(X_{\alb_X},\cL_{\alb_X}^{\otimes r})
	\]
	by cohomology and base change and Kodaira vanishing.  This finishes the proof. 
\end{proof}

 \section{The case of moduli spaces of stable sheaves on abelian surfaces}
	\label{sec:stable-sheaves}

We now turn to Theorem~B.  Our first task is to gather the results of Mukai \cite{Muk2,Muk3}, Yoshioka \cite{Yo} and Riess \cite{Riess} that will be used in its proof.

For a smooth projective surface $S$ with trivial canonical bundle, i.e.~ an abelian or K3 surface, Mukai defines in \cite{Muk2,Muk3} an even unimodular symmetric bilinear form  on the singular cohomology of $S$---now called the Mukai lattice---which plays a crucial role in constructing moduli spaces of stable sheaves on $S$. More concretely, write 
\[
H^{\ev}(S;\Z) \deq \bigoplus_{i=0}^2 H^{2i}(S;\Z)\ \ , \text{ and } \scal{v,w} \deq \int_X -v_0w_4+v_2w_2-v_4w_0\ \ \text{for $v,w\in H^{\ev}(S;\Z)$}\ .
\]
For a coherent sheaf $\cE$ on $S,$ one defines the Mukai vector as $v(\cE) \deq \ch (\cE) \sqrt{\td(S)}$, which simplifies to $\ch (\cE)$ for $S$ abelian.  A Mukai vector is said to be \textit{primitive} if its entries are relatively prime.
 
The moduli space of Gieseker-stable sheaves on $S$ with Mukai vector $v$ with respect to a polarization $H$ will be denoted by ${\cM}_H(v)$, and its Gieseker compactification will be denoted by $\overline{\cM_H(v)}$. We have the following:

\begin{theorem}[Mukai, Yoshioka]\label{thm:MY}
The moduli space ${\cM}_H(v)$ is a smooth quasiprojective variety of dimension $\scal{v,v}+2$. If $H$ lies outside of a certain locally finite set of hyperplanes of the ample cone of $S$ and $v$ is a primitive element of the Mukai lattice of $S$, then $\overline{\cM_H(v)}=\cM_H(v)$ is an irreducible projective manifold and $K_{\cM_{H}(v)} \cong \cO_{\cM_{H}(v)} .$
\end{theorem}
\begin{proof}
The first statement and the existence of a holomorphic symplectic structure on ${\cM}_H(v)$ (which implies the triviality of its canonical bundle) is the content of \cite{Muk3}; see also \cite{HL}, {Theorem 10.4.3}. The compactness of ${\cM}_H(v)$ under the given hypotheses is due to Yoshioka \cite{Yo2}. 
\end{proof}

For the rest of this section we will assume $S$ is an abelian surface.  Recall that for a positive integer $q,$ the Albanese morphism of the Hilbert scheme $S^{[q]}$ parametrizing length-$q$ subschemes of $S$ is the addition map $\sigma_{q}: S^{[q]} \to S$; the \textit{generalized Kummer variety} ${\rm Kum}^{q}(S)$ is defined to be $\sigma_{q}^{-1}(0),$ and it is a compact hyperk\"{a}hler manifold.  Since $\sigma_{q}$ is \'{e}tale-locally trivial, any closed fiber of $\sigma$ is isomorphic to ${\rm Kum}^{q}(S).$ 

Let $v = r+c_1+a\omega\in H^{\ev}(S;\Z)$ be a Mukai vector with $c_1\in \NS(S)$. Following \cite{Yo}, we call $v$ \emph{positive} if $r>0$ (as an element of $H^0(S;\Z) \simeq \Z$) or $r=0$, $c_1$ is effective and $a\neq 0$, or if $r=c_1=0$ and $a<0$. Then we have 

\begin{theorem}[Yoshioka, \cite{Yo}, Theorems 0.1 and 0.2]\label{thm:Yoshioka}
Let $S$ be an abelian surface, $v$  a primitive positive Mukai vector with $c_1(v)\in \NS(S)$ and $\scal{v,v}\geq 6$, $H$ an ample divisor on $S$. Then 
\begin{enumerate}
	\item $\overline{{\cM}_H(v)}={\cM}_H(v)$, and   the Albanese morphism  of ${\cM}_H(v)$ is a smooth fibration over $S \times \widehat{S}$.
	\item If $H$ is sufficiently general, then any fibre $K_H(v)$ of the Albanese morphism of ${\cM}_H(v)$ is a compact hyperk\"{a}hler manifold which is deformation-equivalent to ${\rm Kum}^{\frac{\langle{v,v}\rangle}{2}-1}(S).$
\end{enumerate}	
\end{theorem} 
\begin{proof}
The Albanese morphism is a smooth fibration by \cite{Ka}; all other statements are contained in Theorems 0.1 and 0.2 of \cite{Yo}. 
\end{proof}

\begin{prop}\label{prop:fix-det}
	Assume the hypotheses of Theorem \ref{thm:Yoshioka}, and let $\cQ \in {\rm Pic}^{v_{2}}(S)$ be given.  Then the Albanese map of $\cM_{H}(v,\cQ)$ is a smooth fibration over an abelian surface isogenous to $S.$
\end{prop}

\begin{proof}
	The image of $\cM_{H}(v,\cQ)$ under the Albanese morphism ${\rm alb}_{\cM_{H}(v)} : \cM_{H}(v) \to S \times \widehat{S}$ from (1) of Theorem \ref{thm:Yoshioka} is $S \times \{\eta\}$ for some $\eta \in \widehat{S},$ and so the restriction can be identified with a morphism $\nu : \cM_{H}(v,\cQ) \to S$ whose fibres are all isomorphic to $K_{H}(v),$ and therefore deformation-equivalent to ${\rm Kum}^{\frac{\langle{v,v}\rangle}{2}-1}(S)$ by (2) of Theorem \ref{thm:Yoshioka}.  Given that ${\rm Kum}^{\frac{\langle{v,v}\rangle}{2}-1}(S)$ is compact hyperk\"{a}hler, we have that $K_{H}(v)$ is simply connected, so the Albanese morphism of $\cM_{H}(v,\cQ)$ is necessarily a smooth fibration over an abelian surface isogenous to $S$ as claimed.  
\end{proof}

With this in hand, we obtain effective global generation statements for ample line bundles on the moduli spaces ${\cM}_H(v)$ and $\cM_{H}(v,\cQ).$

\begin{cor}\label{cor:kummer}
	Let $S$ be an abelian surface, $v = (r,c_1,{\rm ch}_{2})$  a primitive positive Mukai vector satisfying $\scal{v,v}\geq 6$, and $H$ a sufficiently general ample divisor on $S$. Then for a positive integer $k$ and an ample line bundle $\cL$ on ${\cM}_H(v)$ or $\cM_{H}(v,\cQ)$, the tensor power $\cL^{\otimes rs}$ is globally generated in codimension $k$ for $r \geq f_{K_{H}(v),k}$ and $s \geq f_{\rm alb_{X\ast}(\cL^{\otimes r}),k}.$  	
\end{cor}
	
\begin{proof}
	Immediate from Theorem~A, Theorem~\ref{thm:Yoshioka}, and Proposition \ref{prop:fix-det}.
\end{proof}

We can improve upon this result by exploiting the special nature of the Beauville-Bogomolov factorization of $\cM_{H}(v),$ namely that we can take the isogeny to be a multiplication map.

\begin{thm}
	\label{thm:mod-sp}
	Let $S$ be an abelian surface, let $v$ be a Mukai vector satisfying ${\langle}v^{2}{\rangle} \geq 6,$ and let $H$ be a very general ample divisor on $S$.  Then for $0 \leq k \leq {\rm dim}~{\cM_{H}(v)}$ we have 
		\begin{equation*}
			f_{\cM_{H}(v),k} \leq \max\biggl\{f_{K_{H}(v),k},\frac{{\langle}v^{2}{\rangle}}{2}+1\biggr\}.
		\end{equation*}
\end{thm}

\begin{proof}
	For simplicity we write $n(v) := \frac{{\langle}v^{2}{\rangle}}{2}.$  From (4.10) and (4.11) in \cite{Yo} we have a Cartesian diagram
	\begin{equation}
		\xymatrix{
			&{K_{H}(v) \times \widehat{A} \times A} \ar[d]_{\overline{\sigma}_{n(v)}} \ar[r]^{\overline{\mu}_{n(v)}} &{{\cM}_{H}(v)} \ar[d]_{\sigma} \\
			&{\widehat{A} \times A} \ar[r]^{\mu_{n(v)}}			&{\widehat{A} \times A}
		}
	\end{equation}
	where $\sigma$ is the Albanese map of $\cM_{H}(v).$  We have that
	\begin{equation*}
		\overline{\sigma_{n(v)}}^{\ast}(\cL^{\otimes m}) \cong \cM^{\otimes m} \boxtimes \mathcal{N}^{\otimes m}
	\end{equation*}  
	
	where $\cM,\mathcal{N}$ are ample line bundles on $K_{H}(v), \widehat{A} \times A$ respectively.  We have that $K_{H}(v)$ is deformation equivalent to the generalized Kummer variety $K_{n(v)-1}$ by (1) of Theorem 0.2 in \cite{Yo}.  By this deformation equivalence, we have that any line bundle $\mathcal{N}$ on $K_{H}(v)$ which arises as a factor of such a pullback is numerically equivalent to an ample divisor of the form ${\mathcal{N}'}^{\otimes n(v)}$ for some ample line bundle ${\mathcal{N}'}$ on $K_{H}(v).$    
	
		\begin{equation*}
		\mu_{n(v)}^{\ast}\sigma_{\ast}(\cL^{\otimes m}) \cong \overline{\sigma_{n(v)\ast}}\overline{\mu_{n(v)}}^{\ast}(\cL^{\otimes m}) 
	\end{equation*}
	\begin{equation*}
		\cong \overline{\sigma_{n(v)\ast}}(\cM^{\otimes m} \boxtimes \mathcal{N}^{\otimes m}) \cong  \overline{\sigma_{n(v)\ast}}(\cM^{\otimes m} \boxtimes \mathcal{N}'^{\otimes mn(v)}) \cong H^{0}(K_{H}(v),\cM^{\otimes m}) \otimes \mathcal{N}'^{\otimes mn(v)}
	\end{equation*}
	
	By Corollary 2.3.6 and Lemma 4.6.2 in \cite{BL}, $\mu_{n}^{\ast}$ acts on ${\rm NS}(A)$ as multiplication by $n^{2}.$  Consequently we have 
	\begin{equation*}
		h^{0}(K_{H}(v),\cM^{\otimes m}) \cdot c_{1}(\mathcal{N}'^{\otimes mn(v)}) = c_{1}(\mu_{n}^{\ast}\sigma_{\ast}(\cL^{\otimes m})) = \mu_{n}^{\ast}c_{1}(\sigma_{\ast}(\cL^{\otimes m})) = n^{2} \cdot c_{1}(\sigma_{\ast}(\cL^{\otimes m}))
	\end{equation*}
	
	Since the rank of $\sigma_{\ast}(\cL^{\otimes m})$ is equal to $h^{0}(K_{H}(v),\cM^{\otimes m})$ we then have that
	
	\begin{equation*}
			\frac{c_{1}(\sigma_{\ast}(\cL^{\otimes m}))}{{\rm rk}~ \sigma_{\ast}(\cL^{\otimes m})} = \frac{c_{1}(\mathcal{N}'^{\otimes mn(v)})}{n^{2}} = \frac{mc_{1}(\mathcal{N}'^{\otimes n(v)})}{n}
	\end{equation*}
	
	By Proposition \ref{prop:pareschi-gg} we will have the global generation of $\sigma_{\ast}(\cL^{\otimes m})$ once we verify that $c_{1}(\sigma_{\ast}(\cL^{\otimes m}) \otimes \mathcal{N}'^{\vee})$ is ample.  We have from the previous calculation that
	
	\begin{equation*}
		\frac{c_{1}(\sigma_{\ast}(\cL^{\otimes m}) \otimes \mathcal{N}'^{\vee})}{{\rm rk}~ (\sigma_{\ast}(\cL^{\otimes m}) \otimes \mathcal{N}'^{\vee})} = \frac{c_{1}(\sigma_{\ast}(\cL^{\otimes m}) \otimes \mathcal{N}'^{\vee})}{{\rm rk}~ \sigma_{\ast}(\cL^{\otimes m})}
	\end{equation*}
	\begin{equation*}
		 = \frac{c_{1}(\sigma_{\ast}(\cL^{\otimes m}))}{{\rm rk}~ \sigma_{\ast}(\cL^{\otimes m})} - c_{1}(\mathcal{N}') = \biggl(\frac{m}{n(v)}-1\biggr) \cdot c_{1}(\mathcal{N}')
	\end{equation*}
	
	Since $m \geq n(v)+1$ by hypothesis, we have the necessary ampleness statement; in particular, $\sigma_{\ast}(\cL^{\otimes m})$ is globally generated, as claimed.  Moreover, since $m \geq f_{K_{H}(v),k}$ by hypothesis, the restriction of $\cL^{\otimes k}$ to any fiber of $\sigma$ is globally generated in codimension $k$, so the evaluation morphism
	\begin{equation*}
		\sigma^{\ast}\sigma_{\ast}(\cL^{\otimes m}) \to \cL^{\otimes m}
	\end{equation*}
	is surjective.  The global generation of $\cL^{\otimes m}$ in codimension $k$ follows at once. 
\end{proof}

We now collect some consequences of this result.  Applying Corollary 4.9 of \cite{Riess}, we have

\begin{cor}
	\label{cor:yoshioka-mov}
	In the setting of Theorem C, we have $f_{\cM_{H}(v),1} \leq \frac{\langle{v^{2}}\rangle}{2}+1.$   \hfill \qedsymbol 
\end{cor}

Also, setting $v = (1,0,-n)$ in Theorem C immediately yields

\begin{cor}
	\label{cor:ab-surf-hilb}
	Let $S$ be an abelian surface, let $n \geq 1$ be an integer, and let $S^{[n]}$ be the Hilbert scheme of length-$n$ subschemes of $S.$  If $K_{n-1}$ is the generalized Kummer variety of $S$ then 
	\begin{equation*}
		f_{S^{[n]},k} \leq \max\{n+1,f_{K_{n-1},k}\}
	\end{equation*}
	In particular, if $\cL$ is an ample line bundle on $S^{[2]}$ then $\cL^{\otimes m}$ is globally generated for all $m \geq 3.$	
\end{cor}

\section{Examples and Complements}
	\label{sec:ex-comp}

The material in this final section will hopefully illuminate, and point towards extensions of, the effective global generation phenomenon discussed earlier in the paper.  First we highlight the importance of our earlier minimality hypothesis on $X$. 

\begin{eg}
	\label{ex:ab-var-blowup}
	Let $A$ be an abelian variety of dimension $n \geq 2$, let $\phi : \widetilde{A} \to A$ be the blow-up of a point $p \in A$ with exceptional divisor $E \cong \P^{g-1},$ and let $L$ be an ample line bundle on $\widetilde{A}$ such that $L|_{E} \cong \cO_{E}(1).$  To see that such an $L$ exists, first consider that if $\phi' : \widetilde{\P^{N}} \to \P^{N}$ is the blow-up at $p' \in \P^{N}$ with exceptional divisor $E' \cong \P^{N-1},$ and $H'$ is the pullback of the hyperplane class on $\P^{N}$ via $\phi',$ then there exists $m' > 0$ such that $L' := m'H'-E'$ is ample and $L'|_{E'} \cong \cO_{E'}(1).$  It follows that if we take an embedding of $A$ in $\P^{N}$ which maps $p$ to $p'$, the desired line bundle $L$ on $\widetilde{A}$ can be taken as the restriction of $L'$ to $\widetilde{A}$ in the induced embedding $\widetilde{A} \subset \widetilde{\P^{N}}.$
	
	Since $K_{\widetilde{A}} = (n-1)E$ and $K_{\widetilde{A}}|_{E} \cong \cO_{E}(1-n)$, we have for all $m \geq 1$ that 
	\begin{equation}
	(K_{\widetilde{A}}+mL)|_{E} \cong \cO_{E}(m+1-n)
	\end{equation}
	hence for $K_{X}+mL$ to be globally generated one must have $m \geq n-1.$  In conclusion, $f_{\widetilde{A}} \geq n-1$ even though $\widetilde{A}$ has maximal Albanese dimension.
\end{eg}

\begin{eg}
	\label{ex:ab-var-f1}
	There exists for each $g \geq 2$ an abelian variety $A$ of dimension $g$ such that $f_{A}=1.$  By \cite{LFG} or \cite{DHS}, there exists a primitive polarization type $\mathbf{d}$ such that for a very general member $(A,L)$ of the moduli space $\cA_{g}(\mathbf{d})$, any line bundle on $A$ algebraically equivalent to $L$ is globally generated.  We can also take $(A,L)$ to be sufficiently general so that ${\rm NS}(A)$ is generated by $c_{1}(L).$  For all such $A,$ we have that $f_{A}=1$ as desired.
\end{eg}

\begin{eg}
	\label{ex:c-y}
	There exists for each odd integer $n \geq 3$ a Calabi-Yau manifold $X$ of dimension $n$ with $f_{X} \geq n+1.$  One such $X$ is a general smooth hypersurface of degree $2n+2$ in the weighted projective space $\P(1^{n},2,n+2)$; see Example 3.2 in 	\cite{Kawa} for details.
\end{eg}

\begin{eg}
	\label{ex:sm-hyp}
	If $X\subseteq\PP^{N}$ is a smooth hypersurface of degree $d \geq 2$ when $N \geq 4$, or a very general hypersurface of degree $d \geq 4$ when $N=3$, then ${\rm Pic}(X)$ is generated by $\cO_{X}(1)$ and $\omega_{X} \cong \cO_{X}(d-N-1)$; therefore 	$f_X=\max\{1,N+1-d\}$.  In particular, $f_X = 1$ for the Calabi-Yau case $d=N+1.$  
\end{eg}

We now estimate $f_X$ in some low-dimensional cases.

 \begin{eg} 
	If $\dim\, X=1$, then $\omega_X\otimes\cA^{\otimes 2}$ is globally generated for every ample line bundle $\cA$, so that $f_X\leq 2$. At the same time, if $P\in X$ is a point and $\cA=\cO_X(P)$ the divisor  $\omega_X\otimes\cA$ is never globally generated, hence $f_X=2$.  
\end{eg}

\begin{eg}
	When $\dim\, X=2,$ Reider's theorem implies $f_X \leq 3.$  However, if we assume further that $X$ is a minimal surface of Kodaira dimension 0, Reider's theorem implies $f_{X} \leq 2.$
\end{eg}

\begin{eg}
If $X$ is a surface which admits a smooth fibration $f\colon X \to C$ over a curve $C,$ the preceding examples imply that $f_X \dleq 3 < f_\pi\cdot f_C$.
\end{eg}

Based on these examples and Theorem A, we offer the following conjecture.

\begin{conj}[Fujita numbers in fibrations]\label{conj:Fujita in fibrations}
	Let $X,Y$ be  smooth projective varieties and   $\pi\colon X\to Y$ a smooth fibration. Then $f_X \dleq f_\pi\cdot f_Y.$ 
\end{conj}


\begin{thebibliography}{RWY}

\bib{AS}{article}{
	AUTHOR = {Angehrn, Urban},
	AUTHOR = {Siu, Yum-Tong},
     	TITLE = {Effective freeness and point separation for adjoint bundles},
   	JOURNAL = {Invent. Math.},
    	VOLUME = {122},
      	YEAR = {1995},
    	NUMBER = {2},
     	PAGES = {291--308}}

\vskip5pt

\bib{Bau}{article}{ 
    	AUTHOR = {Bauer, Thomas},
     	TITLE = {On the cone of curves of an abelian variety},
   	JOURNAL = {Amer. J. Math.},
    	VOLUME = {120},
      	YEAR = {1998},
    	NUMBER = {5},
     	PAGES = {997--1006}
}

\vskip5pt

\bib{BaSz}{article}{
    	AUTHOR = {Bauer, Thomas}
	AUTHOR = {Szemberg, Tomasz},
     	TITLE = {On tensor products of ample line bundles on abelian varieties},
   	JOURNAL = {Math. Z.},
  	VOLUME = {223},
      	YEAR = {1996},
    	NUMBER = {1},
     	PAGES = {79--85},
}

\vskip5pt

\bib{Bea}{article}{
	 AUTHOR = {Beauville, Arnaud},
     	TITLE = {Vari\'et\'es {K}\"ahleriennes dont la premi\`ere classe de {C}hern est nulle},
   	JOURNAL = {J. Differential Geom.},
    	VOLUME = {18},
      	YEAR = {1983},
    	NUMBER = {4},
     	PAGES = {755--782}
}
\vskip5pt

\bib{BL}{book}{
    	AUTHOR = {Birkenhake, Christina},
	AUTHOR={Lange, Herbert},
     	TITLE = {Complex abelian varieties},
    	SERIES = {Grundlehren der Mathematischen Wissenschaften [Fundamental
              Principles of Mathematical Sciences]},
    	VOLUME = {302},
   	EDITION = {2nd ed.},
 	PUBLISHER = {Springer-Verlag, Berlin},
      	YEAR = {2004},
     	PAGES = {xii+635},
}
\vskip5pt

\bib{CP}{article}{
	author={Cao, Junyan},
	author={P\u{a}un, Mihai},
	title={Kodaira dimension of algebraic fiber spaces over abelian
		varieties},
	journal={Invent. Math.},
	volume={207},
	date={2017},
	number={1},
	pages={345--387},
}
\vskip5pt


\bib{CI}{article}{
    	AUTHOR = {Chintapalli, Seshadri},
	AUTHOR = {Iyer, Jaya N. N.},
     	TITLE = {Embedding theorems on hyperelliptic varieties},
   	JOURNAL = {Geom. Dedicata},
    	VOLUME = {171},
      	YEAR = {2014},
     	PAGES = {249--264}
}
\vskip5pt

\bib{DHS}{article}{
    	AUTHOR = {Debarre, Olivier},
	AUTHOR = {Hulek, Klaus},
	AUTHOR = {Spandaw, Jeroen G.},
     	TITLE = {Very ample linear systems on abelian varieties},
   	JOURNAL = {Math. Ann.},
    	VOLUME = {300},
      	YEAR = {1994},
    	NUMBER = {2},
     	PAITGES = {181--202},
}
\vskip5pt

\bib{EL93}{article}{
	author={Ein, Lawrence},
	author={Lazarsfeld, Robert},
	title={Global generation of pluricanonical and adjoint linear series on
		smooth projective threefolds},
	journal={J. Amer. Math. Soc.},
	volume={6},
	date={1993},
	number={4},
	pages={875--903},
}
\vskip5pt

\bib{GP}{article}{
    	AUTHOR = {Gallego, Francisco Javier},
	AUTHOR = {Purnaprajna, B. P.},
     	TITLE = {Very ampleness and higher syzygies for {C}alabi-{Y}au
              threefolds},
   	JOURNAL = {Math. Ann.},
    	VOLUME = {312},
      	YEAR = {1998},
    	NUMBER = {1},
     	PAGES = {133--149},
}
\vskip5pt

\bib{LFG}{article}{
    	AUTHOR = {Garc\'{i}a, Luis Fuentes},
     	TITLE = {A note on the global generation of primitive line bundles on
              abelian varieties},
   	JOURNAL = {Geom. Dedicata},
    	VOLUME = {117},
      	YEAR = {2006},
     	PAGES = {133--135},
}
\vskip5pt

\bib{HPS}{article}{
	author={Hacon, Christopher D.},
	author={Popa, Mihnea},
	author={Schnell, Christian},
	title={Algebraic fiber spaces over abelian varieties: around a recent theorem by Cao and Paun},
	year={2016},
	note={arXiv:1611.08768v2},
}
\vskip5pt


\bib{Heier}{article}{
	author={Heier, Gordon},
	title={Effective freeness of adjoint line bundles},
	journal={Doc. Math.},
	volume={7},
	date={2002},
	pages={31--42},
	issn={1431-0635},
	review={\MR{1911209}},
}
\vskip5pt

\bib{Helmke}{article}{
	author={Helmke, Stefan},
	title={The base point free theorem and the Fujita conjecture},
	conference={
		title={School on Vanishing Theorems and Effective Results in Algebraic
			Geometry},
		address={Trieste},
		date={2000},
	},
	book={
		series={ICTP Lect. Notes},
		volume={6},
		publisher={Abdus Salam Int. Cent. Theoret. Phys., Trieste},
	},
	date={2001},
	pages={215--248},
}
\vskip5pt

\bib{HL}{book}{
    	AUTHOR = {Huybrechts, Daniel},
	AUTHOR =  {Lehn, Manfred},
     	TITLE = {The geometry of moduli spaces of sheaves},
    	SERIES = {Cambridge Mathematical Library},
   	EDITION = {2nd ed.},
 	PUBLISHER = {Cambridge University Press, Cambridge},
      	YEAR = {2010},
     	PAGES = {xviii+325},
}
\vskip5pt


\bib{Ka}{article}{
	author={Kawamata, Yujiro},
	title={Minimal models and the Kodaira dimension of algebraic fiber
		spaces},
	journal={J. Reine Angew. Math.},
	volume={363},
	date={1985},
	pages={1--46},
	doi={10.1515/crll.1985.363.1},
}
\vskip5pt

\bib{Kaw}{article}{
	author={Kawamata, Yujiro},
	title={On Fujita's freeness conjecture for $3$-folds and $4$-folds},
	journal={Math. Ann.},
	volume={308},
	date={1997},
	number={3},
	pages={491--505},
	doi={10.1007/s002080050085},
}
\vskip5pt

\bib{Kawa}{article}{
	AUTHOR = {Kawamata, Yujiro},
     	TITLE = {On effective non-vanishing and base-point-freeness},
   	JOURNAL = {Asian J. Math.},
    	VOLUME = {4},
      	YEAR = {2000},
    	NUMBER = {1},
     	PAGES = {173--181},
}
\vskip5pt

\bib{KM18}{article}{
	 AUTHOR = {K\"{u}ronya, Alex}
	AUTHOR = {Mustopa, Yusuf},
     	TITLE = {Continuous {CM}-regularity of semihomogeneous vector bundles},
   	JOURNAL = {Adv. Geom.},
 	VOLUME = {20},
           YEAR = {2020},
          NUMBER = {3},
          PAGES = {401--412}
}
\vskip5pt

\bib{Lan}{article}{
    	AUTHOR = {Lange, Herbert},
     	TITLE = {Hyperelliptic varieties},
   	JOURNAL = {Tohoku Math. J. (2)},
    	VOLUME = {53},
      	YEAR = {2001},
    	NUMBER = {4},
     	PAGES = {491--510},
}
\vskip5pt

\bib{PAGI}{book}{
	author={Lazarsfeld, Robert},
	title={Positivity in algebraic geometry. I},
	series={Ergebnisse der Mathematik und ihrer Grenzgebiete. 3. Folge. A
		Series of Modern Surveys in Mathematics [Results in Mathematics and
		Related Areas. 3rd Series. A Series of Modern Surveys in Mathematics]},
	volume={48},
	note={Classical setting: line bundles and linear series},
	publisher={Springer-Verlag, Berlin},
	date={2004},
}
\vskip5pt

\bib{PAGII}{book}{
	author={Lazarsfeld, Robert},
	title={Positivity in algebraic geometry. II},
	series={Ergebnisse der Mathematik und ihrer Grenzgebiete. 3. Folge. A
		Series of Modern Surveys in Mathematics [Results in Mathematics and
		Related Areas. 3rd Series. A Series of Modern Surveys in Mathematics]},
	volume={49},
	note={Positivity for vector bundles, and multiplier ideals},
	publisher={Springer-Verlag, Berlin},
	date={2004},
	pages={xviii+385},
}
\vskip5pt

\bib{Muk}{article}{
    	AUTHOR = {Mukai, Shigeru},
     	TITLE = {Semi-homogeneous vector bundles on an {A}belian variety},
   	JOURNAL = {J. Math. Kyoto Univ.},
    	VOLUME = {18},
      	YEAR = {1978},
    	NUMBER = {2},
     	PAGES = {239--272}
}
\vskip5pt


\bib{Muk2}{article}{
	author={Mukai, S.},
	title={On the moduli space of bundles on $K3$ surfaces. I},
	conference={
		title={Vector bundles on algebraic varieties},
		address={Bombay},
		date={1984},
	},
	book={
		series={Tata Inst. Fund. Res. Stud. Math.},
		volume={11},
		publisher={Tata Inst. Fund. Res., Bombay},
	},
	date={1987},
	pages={341--413},
}
\vskip5pt


\bib{Muk3}{article}{
	author={Mukai, Shigeru},
	title={Symplectic structure of the moduli space of sheaves on an abelian
		or $K3$ surface},
	journal={Invent. Math.},
	volume={77},
	date={1984},
	number={1},
	pages={101--116},
}
\vskip5pt

\bib{MR}{article}{
	AUTHOR = {Mukherjee, Jayan}
	AUTHOR = {Raychaudhury, Debaditya},
     	TITLE = {On the projective normality and normal presentation on higher
              dimensional varieties with nef canonical bundle},
   	JOURNAL = {J. Algebra},
 	VOLUME = {540},
      	YEAR = {2019},
     	PAGES = {121--155},
}
\vskip5pt

\bib{Op}{article}{
    	AUTHOR = {Oprea, Dragos},
     	TITLE = {The {V}erlinde bundles and the semihomogeneous {W}irtinger
              duality},
   	JOURNAL = {J. Reine Angew. Math.},
    	VOLUME = {654},
      	YEAR = {2011},
     	PAGES = {181--217},
}
\vskip5pt

\bib{Pa}{article}{
    	AUTHOR = {Pareschi, Giuseppe},
     	TITLE = {Syzygies of abelian varieties},
   	JOURNAL = {J. Amer. Math. Soc.},
    	VOLUME = {13},
      	YEAR = {2000},
    	NUMBER = {3},
     	PAGES = {651--664},
 }
 \vskip5pt

\bib{PP1}{article}{
	author={Pareschi, Giuseppe},
	author={Popa, Mihnea},
	title={Regularity on abelian varieties. I},
	journal={J. Amer. Math. Soc.},
	volume={16},
	date={2003},
	number={2},
	pages={285--302},
}
\vskip5pt

\bib{PP2}{article}{
	AUTHOR = {Pareschi, Giuseppe},
	AUTHOR = {Popa, Mihnea},
     	TITLE = {{$M$}-regularity and the {F}ourier-{M}ukai transform},
   	JOURNAL = {Pure Appl. Math. Q.},
    	VOLUME = {4},
      	YEAR = {2008},
    	NUMBER = {3, Special Issue: In honor of Fedor Bogomolov. Part 2},
     	PAGES = {587--611}
}
\vskip5pt

\bib{Popa}{article}{
	author={Popa, Mihnea},
	title={Verlinde bundles and generalized theta linear series},
	journal={Trans. Amer. Math. Soc.},
	volume={354},
	date={2002},
	number={5},
	pages={1869--1898},
}

\bib{Reider}{article}{
	author={Reider, Igor},
	title={Vector bundles of rank $2$ and linear systems on algebraic
		surfaces},
	journal={Ann. of Math. (2)},
	volume={127},
	date={1988},
	number={2},
	pages={309--316},
}
\vskip5pt

\bib{Riess}{article}{
	AUTHOR = {Riess, Ulrike},
	TITLE = {Base divisors of big and nef line bundles on irreducible symplectic varieties},
	JOURNAL = {Ann. Inst. Fourier},
	NOTE={to appear.}
}
\vskip5pt

\bib{YZ}{article}{
	author={Ye, Fei},
	author={Zhu, Zhixian},
	title={Global generation of adjoint line bundles on projective 5-folds},
	journal={Manuscripta Math.},
	volume={153},
	date={2017},
	number={3-4},
	pages={545--562}
}
\vskip5pt

\bib{Yo}{article}{
    	AUTHOR = {Yoshioka, Kota},
     	TITLE = {Moduli spaces of stable sheaves on abelian surfaces},
   	JOURNAL = {Math. Ann.},
    	VOLUME = {321},
      	YEAR = {2001},
    	NUMBER = {4},
     	PAGES = {817--884}
}
\vskip5pt

\bib{Yo2}{article}{
	author={Yoshioka, Kota},
	title={Chamber structure of polarizations and the moduli of stable
		sheaves on a ruled surface},
	journal={Internat. J. Math.},
	volume={7},
	date={1996},
	number={3},
	pages={411--431}
}

\end{thebibliography}
\end{document}